\documentclass{amsart}
\usepackage{amsmath}
\usepackage{amsfonts}
\usepackage{graphicx}

\newtheorem{theorem}{Theorem}
\theoremstyle{plain}

\newtheorem{corollary}{Corollary}

\newtheorem{definition}{Definition}

\newtheorem{lemma}{Lemma}

\newtheorem{proposition}{Proposition}
\newtheorem{remark}{Remark}

\numberwithin{equation}{section}
\begin{document}
\title[hitting time, decay, arithmetics]{Long hitting time, slow decay of
correlations and arithmetical properties}
\author[S. Galatolo, P. Peterlongo]{S. Galatolo$^1$, P. Peterlongo$^2$}
\address{\textit{$^{1}$Dipartimento di Matematica Applicata, Via Buonarroti
1, Pisa} \\
\indent \textit{$^2$Scuola Normale Superiore, Piazza dei Cavalieri 7, Pisa.}}
\email{s.galatolo@docenti.ing.unipi.it}
\maketitle

\begin{abstract}
Let $\tau _r(x,x_0)$ be the time needed for a point $x$ to enter for the
first time in a ball $B_r(x_0)$ centered in $x_0$, with small radius $r$. We
construct a class of translations on the two torus having particular
arithmetic properties (Liouville components with intertwined denominators of
convergents) not satisfying a logarithm law, i.e. such that for typical $x,x_0$
\begin{equation*}
\liminf_{r\rightarrow 0} \frac{\log \tau _r(x,x_0)}{-\log r} = \infty.
\end{equation*}

By considering a suitable reparametrization of the flow generated by a
suspension of this translation, using a previous construction by Fayad, we
show the existence of a mixing system on three torus having the same
properties.
The speed of mixing of this example must be subpolynomial, because we also show that:
in a system having polynomial decay of correlations,
the $\limsup_{r\to 0}$ of the above ratio of logarithms
(which is also called the upper hitting time indicator) is bounded
from above by a function of the local dimension
and the speed of correlation decay.

More generally, this shows that reparametrizations of torus translations
having a Liouville component cannot be polynomially mixing.
\end{abstract}

\section{Introduction and statement of results}

Let $(X,T,\mu )$ be an ergodic system on a metric space $X$ and fix a point $%
x_0 \in X$. For $\mu$-almost every $x \in X$, the orbit of $x$ goes closer and closer to $x_{0}$ entering (sooner or later) in every positive measure neighborhood of the target point $x_{0}$.

For several applications it is useful to quantify the speed of approaching
of the orbit of $x$ to $x_{0}$. In the literature this has been done in several ways:

\begin{itemize}
\item Hitting time (also called waiting time). Let $B_{r}(x_{0})$ be a ball
with radius $r$ centered in $x_{0}$. We consider the time 
\begin{equation*}
\tau _{r}(x,x_{0})=\min \{n\in \mathbb{N}^{+}:T^{n}(x)\in B_{r}(x_{0})\}
\end{equation*}%
needed for the orbit of $x$ to enter in $B_{r}(x_{0})$ for the first time.
We consider the asymptotic\footnote{%
For two functions of the same variable we intend the asymptotic symbol $f(x)
\sim g(x)$ to mean $\lim \frac{\log f(x)}{\log g(x)}=1$.} behavior of $\tau
_{r}(x,x_{0})$ as $r$ decreases to $0$. Often this is a power law of the
type $\tau _{r}\sim r^{-d}$ and then it is interesting to extract the
exponent $d$ by looking at the behavior of $\frac{\log \tau _{r}(x,x_{0})}{%
-\log r}$ for small $r$. In many systems (having generic artithmetical
properties or fast decay of correlations, see e.g. \cite%
{A,galatolo,G2,kimseo,KiM}) this quantity converges to the local dimension $%
d_{\mu }(x_{0})$\ of $\mu $ at $x_{0}$. We remark that in these systems 
\begin{equation}
\tau _{r}\sim \mu (B_{r}(x_{0}))^{-1}\label{llaw1}.
\end{equation}

\item Logarithm (like) law. Let $d_{n}(x,x_{0})=\min_{1 \leq i\leq n}\mathrm{dist}%
(T^{i}(x),x_{0})$. We consider the asymptotic behavior of $d_{n}(x,x_{0}) $ as $n$ goes to $\infty $. In several examples of flows on suitable spaces (mostly constructed by algebraic means and having fast decay
of correlations), estimates on the behavior for $d_{n}$ (or distances between suitable projections of the points $x,x_0$) are given, in
particular when $x_{0}$ is the point at infinity (see e.g. \cite%
{AM,HV,KM,M,Mas,Su}). These problems are deeply connected with diophantine
approximation and several geometric questions.

\item Dynamical Borel-Cantelli Lemma (Shrinking targets). We consider a
family of balls $B_{i}=B_{r_{i}}(x_{0})$ with $i\in \mathbb{N}$ centered in $x_{0}$ and such that $r_{i}\rightarrow 0$ and we ask if $x\in
\limsup_{i}T^{-i}(B_{i})$ or equivalently $T^{i}(x)\in B_{i}$ for infinitely
many $i$ (other families of decreasing sets have also been considered similarly). Here in some class of systems with fast decay of correlations, or
generic arithmetical properties (see e.g. \cite{CK,Dol,fayad2,GK,kurz}) it is
possible to obtain results like 
\begin{equation}
\sum_{i=1}^{n}\mu (B_{i})=\infty \Rightarrow x\in \limsup_{i}T^{-i}(B_{i})
\label{BC}
\end{equation}
for a.e. $x\in X$.
\end{itemize}

Each of the above points has a large related bibliography which cannot be
cited exhaustively (see references in cited articles). The above points of
view are strictly related. It is easy to see (see Proposition \ref{prop:htll}) that
\begin{equation}
\underset{r\rightarrow 0}{\lim }\frac{\log \tau _{r}(x,x_{0})}{-\log r}%
=\left( \underset{n\rightarrow \infty }{\lim }\frac{-\log d_{n}(x,x_{0})}{%
\log n}\right) ^{-1}  \label{quellasop}
\end{equation}%
when limits exist. Hence the hitting time and the logarithm law approaches
are somewhat equivalent. A statement like equation (\ref{BC}) instead is somewhat slightly
stronger (see \cite{GK} or remark \ref{rem:approcci}).

Given the abundance of relations between the hitting time and measure of the
balls (and hence local dimension), it is worth to look for examples where
there is no such relation. Remark that, by what is said above, such a system
should not have fast decay of correlations or generic arithmetical properties in some sense. In \cite{kimseo} it is proved that if we consider a rotation $x\mapsto x+\alpha :\mathbb{S}^{1}\rightarrow \mathbb{S}^{1}$ and $\alpha $ is a Liouville irrational, then for each $x_{0}$ it holds 
\begin{equation*}
\limsup_{r\rightarrow 0}\frac{\log \tau _{r}(x,x_{0})}{-\log r}=\infty ,
\end{equation*}%
for almost every $x$, while on the other hand, for \emph{every} irrational $\alpha 
$, 
\begin{equation}
\liminf_{r\rightarrow 0}\frac{\log \tau _{r}(x,x_{0})}{-\log r}=1.
\label{kimseo}
\end{equation}%
\noindent Remark that $1$ is equal to the local dimension for every $x_{0}$, hence in this kind of systems a relation like Equation (\ref{llaw1}) does not hold for the limsup behavior but holds for the liminf one.
Trying to generalize this kind of results to the $d$ dimensional torus, it
is natural to ask if there are translations on the $d$-torus violating
the above equality (\ref{kimseo}) with $1$ replaced by $d$ (indeed
this problem  was posed in \cite{T}, section 5). By (\ref{quellasop}) (more precisely, by Proposition \ref{prop:htll}) such translations would also violate the logarithm law 
\begin{equation}
\limsup_{i\rightarrow \infty }\frac{\log d_{i}(x,x_{0})}{-\log i}=\frac{1}{d}%
.  \label{llaw}
\end{equation}
We remark that by Khintchin-Groshev Theorem (see e.g. \cite{T}, section 5
for a statement from our point of view) this equation must be satisfied for
almost every translation on the $d$-torus.

\bigskip 
In section \ref{hitorus} we construct translations on the torus
$\mathbb{T}^{2}$ where logarithm law (\ref{llaw}) does not hold; this
requires particular arithmetic properties. More precisely (see Theorem \ref{thm:infliminf}):

\bigskip \noindent \textbf{Theorem A.} \textit{Let $T_{(\alpha ,\alpha
^{\prime})}$ be the translation on the torus $\mathbb{T}^{2}$ by vector $%
(\alpha,\alpha^{\prime}) \in \mathbb{R}^2$. If $\alpha ,\alpha ^{\prime }$
are irrationals whose continued fraction expansions have denominators of
convergents $q_{n},q_{n}^{\prime }$ which, for some $\gamma >1$, satisfy
(eventually) 
\begin{equation*}
q_{n}^{\prime }\geq q_{n}^{\gamma }~;~q_{n+1}\geq q_{n}^{\prime }{}^{\gamma }
\end{equation*}
then for each $x_{0}$ 
\begin{equation*}
\liminf_{r\rightarrow 0}\frac{\log \tau _{r}(x,x_{0})}{-\log r}\geq \gamma
\end{equation*}
holds for almost each $x\in \mathbb{T}^{2}$.}\bigskip

This answers negatively the above mentioned problem of \cite{T}
(see also \textsc{question 3} in \cite{tseng2}).

In \cite{fayad2} Fayad provides an example of a mixing system which does not
have the monotone shrinking target property. In such system relation (\ref%
{BC}) is not satisfied for some decreasing sequence of balls. The example is
given by a reparametrization of a translation flow having arithmetical
properties which are included in the class considered above. Hitting time
indicators are well behaved under bounded reparametrizations and Poincar\'{e}
sections (see section \ref{fow}). By these properties, using Fayad's
construction and Theorem A, we can strenghten the result proved in \cite%
{fayad2}, by the following (see Corollary \ref{thm:fydconn})

\bigskip \noindent \textbf{Theorem B.} \textit{There exists a smooth, mixing system $%
(\mathbb{T}^{3},T,\mu) $ on the three dimensional torus, with absolutely
continuous invariant measure $\mu $ (and continuous positive density) such that for
every $x_{0}$} 
\begin{equation*}
\liminf_{r\rightarrow 0}\frac{\log \tau _{r}(x,x_{0})}{-\log r}=\infty
\end{equation*}%
\emph{holds for }$\mu $\emph{\ almost every }$x\in \mathbb{T}^{3}.$\bigskip

We remark that this is slightly stronger than what is proved in \cite{fayad2}
because, as said above, the shrinking target property implies logarithm law.
In this example the speed of correlation decay (see Definition \ref{sup})
must be lower than any power law.

This fact follows from the general result proved in the last section of the
paper: there is a relation between local dimension, decay of correlations
and hitting time. Several results proving a relation between hitting time
and ball's measure need some sort of rapid decay of correlations. Polynomial
decay is enough to prove an upper bound for the hitting time. We state this
in the following (for a more general result see Theorem \ref{1}):

\bigskip \noindent \textbf{Theorem C.} \textit{If a system on a manifold of
dimension $d$ has absolutely continuous invariant probability measure
with continuous and strictly positive density (the assumption on the measure can be largely
relaxed) and polynomial decay of correlations (on Lipschitz observables)
with exponent $p$, then for each $x_{0}$ 
\begin{equation*}
d \leq \underset{r\rightarrow 0}{\lim \sup }\frac{\log \tau _{r}(x,x_{0})}{%
-\log r}\leq d + \frac{2d+2}{p}
\end{equation*}%
holds for $\mu $-almost each $x$.}\bigskip

This theorem and invariance of hitting time under positive speed
reparametrizations (Proposition \ref{reparametriz}) give a more general
upper bound on the decay of correlations for reparametrizations of torus
translations, depending on its arithmetical properties (see Corollary \ref%
{cormanifold}):

\bigskip \noindent \textbf{Theorem D.} \textit{If $(\mathbb{T}^{d},T)$ is
the time-1 map of any reparametrization with positive speed of an irrational
translation flow with direction $\alpha =(\alpha _{1},...,\alpha _{d})$ and
some of the $\alpha _{i}$ has type $\gamma > d$ (see section \ref{hitorus}
for definiton of type) then the speed of decay of correlations of this
system is slower than a power law with exponent ${\frac{2d+2}{\gamma - d}}$.}%
\bigskip

In particular a reparametrization of a translation with an angle having
infinite type must have \emph{subpolynomial} decay of correlations. We
remark (see \cite{F2}) that there are polynomially mixing reparametrizations
whith finite type angles (and no strictly positive speed).

This paper is organized as follows: in section \ref{sec2} we give basic
definitions and state some general facts about hitting time, local dimension
and relations between them. In section \ref{newsec} we recall basic facts on
continued fractions and review known results about type and circle
rotations. In section \ref{hitorus} we prove Theorem A (Theorem \ref%
{thm:infliminf}). In section \ref{fow} we state some easy relations about
hitting time in flows and its behavior under reparametrizations, time one
map, and poincare sections, allowing to deduce our Theorem B from Theorem A
and Fayad's result\ ( see Corollary \ref{thm:fydconn}). In section \ref%
{lastsection} we prove a result on speed of decay of correlation and hitting
time (Theorem \ref{1}); this will imply Theorem C and Theorem D. Theorem A
and Theorem D together imply that Fayad's example has subpolynomial decay of
correlations.

In the Appendix we give some auxiliary proposition which are useful to
relate some of the different point of views on the subject mentioned at the
beginning of this introduction.

\bigskip \noindent \textbf{Acknowledgements.} We would like to thank Bassam
Fayad, for stimulating discussions and for pointing us some relevant papers.
Second author would like to thank SISSA/ISAS (Trieste) where a first part of
research work was carried; he would like also to thank his parents for
financial support during last part of the work.

\section{Hitting time and local dimension: basic facts\label{sec2}}

Let $(X,T)$ be a discrete time dynamical system where $X$ is a separable
metric space equipped with a Borel finite measure $\mu $ and $T:X\rightarrow
X$ is a measurable map.

\begin{definition}
The first entrance time of the orbit of $x$ in the ball $B(x_{0},r)$ with
center $x_{0} $ and radius $r$ is 
\begin{equation*}
\tau _{r}(x,x_{0})=\min \{n\in \mathbb{N^{+}}:T^{n}(x)\in B(x_{0},r)\}.
\end{equation*}
\end{definition}

We want to study the behaviour of $\tau _{r}(x,x_{0})$ as $r\rightarrow 0$.
In many interesting cases this is a power law $\tau _{r}(x,x_{0})\sim r^{-R}$.
In order to extract the exponent, we define

\begin{definition}
The upper and lower \emph{hitting time indicators} are 
\begin{equation}
\overline{R}(x,x_{0})=\limsup_{r\rightarrow 0}\frac{\log \tau _{r}(x,x_{0})}{%
-\log r},\quad \underline{R}(x,x_{0})=\liminf_{r\rightarrow 0}\frac{\log
\tau _{r}(x,x_{0})}{-\log r}.  \label{expo}
\end{equation}
\end{definition}

If for some $r,$ $\tau _{r}(x,x_{0})$ is not defined then $\overline{R}%
(x,x_{0})$ and $\underline{R}(x,x_{0})$ are set to be equal to infinity. We
remark that the indicators $\overline{R}(x)$ and $\underline{R}(x)$ of
quantitative recurrence defined in \cite{BS} are obtained as a special case, 
$\overline{R}(x)=\overline{R}(x,x)$, $\underline{R}(x)=\underline{R}(x,x)$.

We recall some basic properties of $R(x,x_{0})$ which follow from the
definition:

\begin{proposition}\label{inizz} $\overline{R}(x,x_{0})$ satisfies the following properties
(and the same is true with $\overline{R}$ replaced by $\underline{R}$):
\begin{itemize}
\item $x\notin T^{-1}(x_{0})$ implies $\overline{R}(x,x_{0})=\overline{R}(T(x),x_{0})$.

\item If $F$ is bilipschitz and%
\begin{equation}
\begin{array}{ccc}
X & \overset{F}{\longrightarrow } & Y \\ 
\downarrow _{T_{1}} &  & \downarrow _{T_{2}} \\ 
X & \overset{F}{\longrightarrow } & Y%
\end{array}%
\end{equation}%
commutes, then $\overline{R}(x,x_{0})=\overline{R}(F(x),F(x_{0}))$.

\item If $T$ is bilipschitz, and $x\notin T^{-1}(T(x_{0}))$ then $\overline{R}(x,x_{0})=\overline{R}(x,T(x_{0}))$.
\end{itemize}
\end{proposition}

By the above properties, if $T$ is an ergodic isometry on the $d$-torus then 
$\overline{R}(x,x_{0})=\overline{R},$ $\underline{R}(x,x_{0})= $ $\underline{%
R}$ are constant for almost every $x,x_{0}$.

The continous limit in the definition of hitting time indicator can be
reduced to a discrete limit:

\begin{lemma}
\label{lemmino} Let $r_{n}$ be a decreasing sequence of reals, such that $%
r_{n}\rightarrow 0$. Suppose that there is a constant $1>c>0$ satisfying $%
r_{n+1}>cr_{n}$ eventually as $n$ increases. If $\tau _{r}:\mathbb{R}%
\rightarrow \mathbb{R}$ is decreasing, then $\liminf_{n\rightarrow \infty }%
\frac{\log \tau _{r_{n}}}{-\log r_{n}}=\liminf_{r\rightarrow 0}\frac{\log
\tau _{r}}{-\log r}$ and $\limsup_{n\rightarrow \infty }\frac{\log \tau
_{r_{n}}}{-\log r_{n}}=\limsup_{r\rightarrow 0}\frac{\log \tau _{r}}{-\log r}
$.
\end{lemma}

\begin{proof}
If $r_{n}\geq r\geq r_{n+1}\geq cr_{n}$ then $\tau _{r_{n+1}}\geq \tau
_{r}\geq \tau _{r_{n}},$ moreover $\log r_{n}\geq \log r\geq \log
r_{n+1}\geq \log cr_{n}\geq \log cr...$ hence for $n$ big enough%
\begin{equation*}
\frac{\log \tau _{r_{n+1}}}{-\log r_{n+1}}\geq \frac{\log \tau _{r}}{-\log
r-\log c}\geq \frac{\log \tau _{r_{n}}}{-\log r_{n}-2\log c},
\end{equation*}%
which gives the statement by taking the limits.
\end{proof}

We now recall some definitions about local dimension. If $X$ is a metric
space and $\mu $ is a measure on $X$ the local dimension of $\mu $ at $x_0$ is
defined as follows

\begin{definition}
The upper and lower \emph{local dimensions} at $x_0$ are 
\begin{equation*}
\overline{d}_{\mu }(x_0)=\limsup_{r\rightarrow 0}\frac{\log \mu (B_{r}(x_0))}{%
\log r},\quad \underline{d}_{\mu }(x_0)=\liminf_{r\rightarrow 0}\frac{\log \mu
(B_{r}(x_0))}{\log r}
\end{equation*}
\end{definition}

If the limit exists we denote its value as $d_{\mu }(x_0)=\lim_{r\rightarrow 0}%
\frac{\log \mu (B_{r}(x_0))}{\log r}$. In this case $\mu (B_{r}(x_0))\sim
r^{d_{\mu }(x_0)}$. If $\overline{d}_{\mu }(x_0)=\underline{d}_{\mu }(x_0)=d$
almost everywhere the system is called exact dimensional. In this case many
notions of dimension of a measure coincide. In particular $d$ is equal
to the infimum Hausdorff dimension of full measure sets: $d=\inf \{\dim _{H}Z:\mu (Z)=1\}$. This
happens in a large class of systems, for example in systems having nonzero
Lyapunov exponents almost everywhere (see the book \cite{P},chapter 2).

In general measure preserving systems it is rather easy to prove that
behavior of the hitting time is related to the invariant measure of the
system.

\begin{proposition}(\cite{G})
\label{GAN} If $(X,T,\mu )$ is a dynamical system over a separable
metric space, with an invariant measure $\mu $ (not necessarily finite),
then for each $x_{0}$%
\begin{equation}
\underline{R}(x,x_{0})\geq \underline{d}_{\mu }(x_{0})\ ,\ \overline{R}%
(x,x_{0})\geq \overline{d}_{\mu }(x_{0})  \label{llrels}
\end{equation}%
holds for $\mu $-almost every $x$.
\end{proposition}

\begin{remark}
\label{rem:tausumu}  Relations of this type can appear (see e.g.
statement of Theorem \ref{1}) in the form of a direct logarithm law between
waiting time and measure; for example 
\begin{equation*}
\liminf_{r\rightarrow 0}\frac{\log \tau _{r}(x,x_{0})}{-\log \mu
(B_{r}(x_{0}))}\geq 1
\end{equation*}%
implies (but it is slightly more precise of) first inequality in (\ref%
{llrels}) while analogously 
\begin{equation*}
\limsup_{r\rightarrow 0}\frac{\log \tau _{r}(x,x_{0})}{-\log \mu
(B_{r}(x_{0}))}\leq 1
\end{equation*}%
implies the upper bound for hitting time: $\overline{R}(x,x_{0})\leq 
\overline{d}_{\mu }(x_{0})$. Implications become equivalences when
exact-dimensional measure is assumed.
\end{remark}

In systems with decay of correlations (see Defintion \ref{sup}) faster than
any power law the above inequalities become equalities (for precise
statementes, see \cite{galatolo} or Theorem \ref{1}).

\section{Continued fractions, type and circle rotations\label{newsec}}

We briefly recall the basic definitions and properties of continued
fractions ( for a general reference see e.g. \cite{K}) that will be needed
in the sequel. Let $\alpha $ be an irrational number, and denote by $%
[a_{0};a_{1},a_{2},\ldots ]$ its continued fraction expansion: 
\begin{equation*}
\alpha =a_{0}+\cfrac{1}{a_1 + \cfrac{1}{a_2 + \ldots}}=:[a_{0};a_{1},a_{2},%
\ldots ].
\end{equation*}%
The integers $a_{0},a_{1},a_{2}\ldots $ are called partial quotients of $%
\alpha $ and are all positive except for $a_{0}$. As usual, we define
inductively the sequences $p_{n}$ and $q_{n}$ by: 
\begin{eqnarray*}
p_{-1}=1, &p_{0}=0,&p_{k+1}=a_{k+1}p_{k}+p_{k-1}\text{ for }k\geq 0; \\
q_{-1}=0, &q_{0}=1,&q_{k+1}=a_{k+1}q_{k}+q_{k-1}\text{ for }k\geq 0.
\end{eqnarray*}%
The fractions $p_{n}/q_{n}$ are called the \emph{convergents} of $\alpha $,
as they do in fact converge to it. Moreover they can be seen as \emph{best
approximations} of $\alpha $ in the following sense. Denote by $\Vert x\Vert
:=\min_{n\in \mathbb{Z}}|x-n|$ the distance of a real number form the
integers. Then $q=q_{n}$ for some $n$ if and only if 
\begin{equation*}
\Vert q\alpha \Vert <\Vert q^{\prime }\alpha \Vert \text{ for every positive 
}q^{\prime }<q
\end{equation*}%
and $p_{n}$ is the integer such that $\Vert q_{n}\alpha \Vert =|q_{n}\alpha
-p_{n}|$.

\begin{proposition}
\label{prop:fibo} For any irrational number $\alpha $ we have that the
sequence of its convergents $q_{n}$ eventually grows faster than any fixed
power of $n$.
\end{proposition}

\begin{proof}
$q_{n+1}\geq q_{n}+q_{n-1}$ so $q_{n}\geq f_{n}$ where $f_{n}$ is Fibonacci
sequence.
\end{proof}

\begin{proposition}
(\cite{K}, Thm. 9 and Thm. 13) \label{prop:bounds} 
\begin{equation*}
\frac{1}{2}\frac{1}{q_{n+1}} < \frac{1}{q_n + q_{n+1}} < \Vert q_n \alpha
\Vert < \frac{1}{q_{n+1}}
\end{equation*}
\end{proposition}

To measure how well an irrational number is approximated by rational numbers
one introduces the notion of type.

\begin{definition}
\label{def:type} The \emph{type} (or Diophantine exponent\footnote{%
Type plus one is also called irrationality exponent, see for example \cite%
{sondow}.}) of an irrational number $\alpha$ is defined as 
\begin{equation*}
\gamma (\alpha ):=\sup \{\beta :\liminf_{q\rightarrow \infty }q^{\beta
}\Vert q\alpha \Vert =0\}
\end{equation*}
\end{definition}

In terms of the $q_n$s, it is easy to show that this is equal to 
\begin{equation*}
\limsup_{n\rightarrow \infty }\frac{\log q_{n+1}}{\log q_{n}}
\end{equation*}
(a proof can be given using Propositions \ref{prop:didienne}, \ref%
{prop:bosheq} and \ref{prop:bounds}).

Every irrational number has type $\geq 1$. The set of number of type $1$
(also known as Roth type) is of full measure; the set of numbers of type $%
\gamma $ has Hausdorff dimension $\frac{2}{\gamma +1}$. There exist numbers
of infinite type, called \emph{Liouville} numbers; their set is dense and
uncountable and has zero Hausdorff dimension.

We consider a rotation of the circle as a translation on $\mathbb{S}^1 \cong 
\mathbb{R}/\mathbb{Z}$ by a number $\alpha \in \mathbb{R}$ and we denote it
by $T_\alpha : x \mapsto x + \alpha \pmod 1$. For an irrational rotation the
following relations are known between hitting time indicator of $T_\alpha$
and the type of the rotation number $\gamma(\alpha)$.

\begin{theorem}
(\cite{kimseo}) For a fixed point $x_0$ and for Lebesgue-almost every $x$ 
\begin{equation*}
\overline{R}(x,x_0)=\gamma, \qquad \underline{R}(x,x_0)=1.
\end{equation*}
\end{theorem}

This result was preceded by an analogous one involving recurrence time.

\begin{theorem}
(\cite{choeseo}) For every $x$ 
\begin{equation}
\underline{R}(x,x)=\frac{1}{\gamma },\qquad \overline{R}(x,x)=1.
\label{eq:recind}
\end{equation}
\end{theorem}

Notice that for each translation $\underline{R}(x,x_{0})$ is almost
everywhere equal to the dimension of the circle. It is natural to ask if
this generalizes to translations on $n$-tori. In next section we will show
that this does not hold even for $n=2$.

\section{Long hitting time for torus translations\label{hitorus}}

We denote by $T_{(\alpha ,\alpha^{\prime })}$ the translation by vector
$(\alpha,\alpha^{\prime}) \in \mathbb{R}^2$ on the torus
$\mathbb{T}^{2} \cong \mathbb{R}^{2}/\mathbb{Z}^{2}$;
we write $\tau_{r}^{(\alpha ,\alpha ^{\prime})}(x)$ for the hitting time needed
for a point $x \in \mathbb{T}^{2}$ to reach the ball of radius $r$ centered in
a fixed point $x_0$, where the metric is
given by sup distance.
This transformation is the direct product of two
rotations of the circle ($T_{\alpha },T_{\alpha ^{\prime }}$) with hitting
time functions $\tau _{r}^{\alpha }(x),\tau _{r}^{\alpha ^{\prime}}(x)$ as above.
The notation for continued fractions is easily guessed (e.g. $q_{n}^{\prime
} $ are denominators of convergents of $\alpha ^{\prime }$) and $\gamma
,\gamma ^{\prime }$ are respectively the types of $\alpha $ and $\alpha
^{\prime }$. In this section $x_0$ is assumed to be $0$, for circle case, or $(0,0)$,
for the torus case. This can be done without loss of generality regarding statements about
hitting time indicator (see Proposition \ref{inizz}) .

As in the case of one dimensional rotations, the hitting time behavior of $%
T_{(\alpha ,\alpha ^{\prime })}$ will depend on the arithmetical properties
of $(\alpha ,\alpha ^{\prime })$. The following is an easy (and not sharp at
all) estimation for the limsup indicator

\begin{proposition}
\label{UNO} If $T_{(\alpha ,\alpha ^{\prime })}$ is a translation on the two
torus as above and $\gamma ,\gamma ^{\prime }$ are respectively the type of $%
\alpha $ and $\alpha ^{\prime }$, then for each $x_{0}$ 
\begin{equation*}
\overline{R}(x,x_{0})\geq \max (\gamma ,\gamma ^{\prime }).
\end{equation*}
\end{proposition}

\begin{proof}
Since the distance on the torus is the $\sup $ distance then $%
T_{(\alpha,\alpha^{\prime})}^{n}(x)$ is near $x_{0}$ only if both
coordinates are, so $\tau _{r}^{(\alpha ,\alpha ^{\prime })}(x,x_{0})\geq
\max (\tau _{r}^{(\alpha )}(x,x_{0}),\tau _{r}^{(\alpha ^{\prime
})}(x,x_{0}))$ and the statement follows straigthforwardly by this.
\end{proof}

The above proposition can be trivially generalized to the $n$ torus. This
proposition implies that if one of the angles has infinite type then the $%
\limsup$ indicator of the whole translation is infinite.

The key to obtain non trivial lower bound for $\liminf$ indicator is to
consider irrationals with intertwined denominators of convergents. Take $%
\gamma >1$ and let $Y_{\gamma }\subset \mathbf{R}^{2}$ be the class of
couples of irrationals $(\alpha ,\alpha ^{\prime })$ given by the following
conditions on their convergents to be satisfied eventually: 
\begin{equation*}
q_{n}^{\prime }\geq q_{n}^{\gamma };
\end{equation*}%
\begin{equation*}
q_{n+1}\geq q_{n}^{\prime }{}^{\gamma }.
\end{equation*}%
We note that each $Y_{\gamma }$ is uncountable and dense in $[0,1]\times
\lbrack 0,1]$ and each irrational of the couple is of type at least $\gamma
^{2}$. The set $Y_\infty= \bigcap_\gamma Y_{\gamma}$ is also uncountable and
dense in unit square and both coordinates of the couple are Liouville
numbers (cfr. with construction appearing in Fayad's example, see Theorem %
\ref{fayadex}).

Our main result in this section is the following:

\begin{theorem}
\label{thm:infliminf} If $T_{(\alpha ,\alpha ^{\prime })}$ is a translation
of the two torus by a vector $(\alpha ,\alpha ^{\prime })\in Y_{\gamma }$
and $x_{0}\in \mathbb{T}^{2}$, then for Lebesgue-almost every $x\in \mathbb{T%
}^{2}$ 
\begin{equation*}
\underline{R}(x,x_{0}) \geq \gamma .
\end{equation*}
In particular, for $(\alpha,\alpha^{\prime}) \in Y_\infty$ almost everywhere 
$\underline{R}(x,x_{0})=\infty$.
\end{theorem}

We will prove this long hitting time behaviour reducing to the
one-dimensional case: Lemma \ref{lem:keyl} locates in one dimension the
radii where hitting time is not so long, i.~e.~`moments' when orbit is near
the target; alternating character of the convergents of $\alpha$ and $%
\alpha^{\prime}$ implies that when one coordinate is near the target the
other is far from it and let us deduce long hitting time behaviour for the
torus.

For circle rotations the following proposition gives the relation between
measure of points with fixed hitting time and convergents of the continued
fraction, and will be used to prove our key Lemma. We omit the easy proof
(also follows from Proposition 6 of \cite{kimseo}).

\begin{proposition}
\label{prop:enough} Given $r$ such that $\Vert q_n \alpha \Vert < 2r \leq
\Vert q_{n-1} \alpha \Vert$ we have 
\begin{equation*}
\begin{array}{ll}
\mu \lbrace x: \tau_r^\alpha(x)=k \rbrace = 2r & \quad \mathit{for} \quad k \leq
q_n; \\ 
\mu \lbrace x: \tau_r^\alpha(x)=k \rbrace \leq \Vert q_n \alpha \Vert & \quad 
\mathit{for} \quad k > q_n.%
\end{array}%
\end{equation*}
\end{proposition}

\begin{lemma}
\label{lem:keyl} 
Let $\beta, M, N \geq 1$. Taken $r$ such that for some $n$ we have that 
\begin{equation*}
\Vert q_n \alpha \Vert < M \Vert q_n \alpha \Vert^{1/\beta} \leq 2r \leq 
\frac{1}{N} \Vert q_{n-1} \alpha \Vert < \Vert q_{n-1} \alpha \Vert
\end{equation*}
then 
\begin{equation*}
\mu \lbrace x : \tau_r^\alpha(x) < (2r)^{-\beta} \rbrace \leq
\frac{1}{N} + \frac{1}{M^\beta}.
\end{equation*}
\end{lemma}

\begin{proof}
From hypothesis 
\begin{equation*}
(2r)^{-\beta} \leq \frac{1}{M^\beta \Vert q_n \alpha \Vert},
\end{equation*}
thus 
\begin{equation*}
\lbrace x: \tau_r^\alpha(x) < (2r)^{-\beta} \rbrace \subset
\lbrace x: \tau_r^\alpha(x) < \frac{1}{M^\beta \Vert q_n \alpha \Vert} \rbrace.
\end{equation*}
Last set can be thought as a finite union and its measure given by
\begin{equation*}
\sum_{k \in \mathbb{N}, 1 \leq k \leq q_n} \mu \lbrace x: \tau_r^\alpha(x) = k \rbrace
\quad + \quad \sum_{k \in \mathbb{N}, q_n < k < \frac{1}{M^\beta \Vert q_n \alpha \Vert}}
       \mu \lbrace x: \tau_r^\alpha(x) = k \rbrace
\end{equation*}
so Proposition \ref{prop:enough} and Proposition \ref{prop:bounds} imply: 

\begin{eqnarray*}
\mu \lbrace \tau_r^\alpha < \frac{1}{M^\beta \Vert q_n \alpha \Vert} \rbrace & \leq &
2r \cdot q_n + \Vert q_n \alpha \Vert \cdot \frac{1}{M^\beta \Vert q_n\alpha \Vert} \leq 
\frac{1}{N} \Vert q_{n-1} \alpha \Vert q_n + \frac{1}{M^\beta} \leq \\
& \leq & \frac{1}{N} + \frac{1}{M^\beta}.
\end{eqnarray*}
\end{proof}

\begin{proof}
(of Theorem \ref{thm:infliminf}) We denote by $\mu^2$ the Lebesgue measure
on $\mathbb{T}^2$ and by $\mu$ the Lebesgue measure on $\mathbb{S}^1$.

Since we consider a ratio of logarithms, the $\liminf $ is preserved by
considering the limit along the sequence $r_{i}:=e^{-i}$ (see lemma \ref%
{lemmino}).

Take $\beta <\gamma $. Our interest lies in this sequence of subsets of the
two-torus: 
\begin{equation*}
\overline{A}_{i}:=\{x \in \mathbb{T}^2:\tau _{r_{i}}^{(\alpha ,\alpha ^{\prime })}(%
x)<(2r_{i})^{-\beta }\}=\{x \in \mathbb{T}^2:\frac{\log \tau
_{r_{i}}^{(\alpha ,\alpha ^{\prime })}(x)}{-\log r_{i}}+%
\frac{\beta \log 2}{-\log r_{i}}<\beta \}
\end{equation*}%
if we prove that the measures of $\overline{A}_{i}$ are summable,
Borel-Cantelli lemma will imply that the measure of their set-theoretic $%
\limsup $ (set of points that fall infinitely often in the sequence) is
zero. Thus we have 
\begin{equation*}
\mu^2 \{\liminf_{i}\frac{\log \tau _{r_{i}}^{(\alpha ,\alpha ^{\prime })}
}{-\log r_{i}}<\beta \}\leq \mu^2 (\limsup_{i}\{\frac{%
\log \tau _{r_{i}}^{(\alpha ,\alpha ^{\prime })}}{-\log
r_{i}}<\beta \})=0
\end{equation*}%
The thesis follows taking $\beta $ arbitrarily near to $\gamma $.

We reduce the two dimensional problem to a one-dimensional one by defining
the following subsets of the circle 
\begin{equation*}
A_{i}:=\{x \in \mathbb{S}^1:\tau _{r_{i}}^{\alpha }(x)<(2r_{i})^{-\beta }\}
\end{equation*}%
\begin{equation*}
A_{i}^{\prime }:=\{x \in \mathbb{S}^1:\tau _{r_{i}}^{\alpha ^{\prime }}(x)<(2r_{i})^{-\beta
}\}
\end{equation*}%
and observing that $\mu ^{2}(\overline{A}_{i})\leq \min (\mu (A_{i}),\mu
(A_{i}^{\prime }))$.

To prove that $\overline{A}_{i}$ are summable we will sometimes bound its
measure with $\mu (A_{i})$, sometimes with $\mu (A_{i}^{\prime })$. In fact,
we will prove that two appropriate subsequences of $A_{i}$ and $%
A_{i}^{\prime }$ are summable and that the union of the indexes of the
subsequences covers a neighbourhood of infinity.

The sets of indexes we take are, respectively, $\bigcup_{n}I_{n}$ and $%
\bigcup_{n}I_{n}^{\prime }$ where $I_{n}$ and $I_{n}^{\prime }$ are used to
group subsets of consecutive indexes: 
\begin{equation*}
I_{n}:=\{i>0:M_{n}\Vert q_{n}\alpha \Vert ^{\frac{1}{\beta }}\leq 2r_{i}\leq 
\frac{1}{N_{n}}\Vert q_{n-1}\alpha \Vert \},
\end{equation*}%
\begin{equation*}
I_{n}^{\prime }:=\{i>0:M_{n}^{\prime }\Vert q_{n}^{\prime }\alpha ^{\prime
}\Vert ^{\frac{1}{\beta }}\leq 2r_{i}\leq \frac{1}{N_{n}^{\prime }}\Vert
q_{n-1}^{\prime }\alpha ^{\prime }\Vert \};
\end{equation*}%
this particular choice of $M_n, M_n^{\prime},N_n, N_n^{\prime}$ will serve
our purposes 
\begin{equation}  \label{eq:mncondition}
M_{n}=n ,M_{n}^{\prime }=n\quad N_{n}=\frac{1}{M_{n-1}^{\prime }}\frac{\Vert
q_{n-1}\alpha \Vert }{\Vert q_{n-1}^{\prime }\alpha ^{\prime }\Vert ^{\frac{1%
}{\beta }}}\quad N_{n}^{\prime }=\frac{1}{M_{n}}\frac{\Vert q_{n-1}^{\prime
}\alpha ^{\prime }\Vert }{\Vert q_{n}\alpha \Vert ^{\frac{1}{\beta }}}.
\end{equation}

If the two sequences of intervals 
\begin{equation}
\lbrack M_{n}\Vert q_{n}\alpha \Vert ,\frac{1}{N_{n}}\Vert q_{n-1}\alpha
\Vert ],\quad \lbrack M_{n}^{\prime }\Vert q_{n}^{\prime }\alpha ^{\prime
}\Vert ^{\frac{1}{\beta }},\frac{1}{N_{n}^{\prime }}\Vert q_{n-1}^{\prime
}\alpha ^{\prime }\Vert ]  \label{eq:intervals}
\end{equation}%
are eventually non-empty then they cover a neighbourhood of zero from the right
(because they alternate and equation (\ref{eq:mncondition}) forces them to have equal
extremes or overlap), thus $\bigcup_{n}I_{n}\cup I_{n}^{\prime }$ covers a
neighbourhood of infinity.

Now we use the fact that $(\alpha ,\alpha ^{\prime })\in Y_{\gamma }$ and
Proposition \ref{prop:bounds} to show that the ratio between the extremes of
the intervals in (\ref{eq:intervals}) eventually grows to infinity (being in
particular bigger than 1 and then forcing intervals to be not empty): 
\begin{equation*}
\frac{\frac{1}{N_{n}}\Vert q_{n-1}\alpha \Vert }{M_{n}\Vert q_{n}\alpha
\Vert ^{\frac{1}{\beta }}}=\frac{M_{n-1}^{\prime }\Vert q_{n-1}^{\prime
}\alpha ^{\prime }\Vert ^{\frac{1}{\beta }}}{M_{n}\Vert q_{n}\alpha \Vert ^{%
\frac{1}{\beta }}}\geq \frac{n-1}{n}\left( \frac{q_{n+1}}{2q_{n}^{\prime }}%
\right) ^{\frac{1}{\beta }}\geq \frac{n-1}{n}\left( \frac{{q_{n}^{\prime }}%
^{\gamma }}{2{q_{n}^{\prime }}}\right) ^{\frac{1}{\beta }}\rightarrow \infty
\end{equation*}%
\begin{equation*}
\frac{\frac{1}{N_{n}^{\prime }}\Vert q_{n-1}^{\prime }\alpha ^{\prime }\Vert 
}{M_{n}^{\prime }\Vert q_{n}^{\prime }\alpha ^{\prime }\Vert ^{\frac{1}{%
\beta }}}=\frac{M_{n}\Vert q_{n}\alpha \Vert ^{\frac{1}{\beta }}}{%
M_{n}^{\prime }\Vert q_{n}^{\prime }\alpha ^{\prime }\Vert ^{\frac{1}{\beta }%
}}\geq \left( \frac{q_{n+1}^{\prime }}{2q_{n+1}}\right) ^{\frac{1}{\beta }%
}\geq \left( \frac{q_{n+1}^{\gamma }}{2q_{n+1}}\right) ^{\frac{1}{\beta }%
}\rightarrow \infty
\end{equation*}

To prove summability we need a trick to apply Lemma \ref{lem:keyl} to
the whole bunch of $r_i$ contained in a single set $I_n$. This is based on a
simple remark: if $2r \in [M \Vert q_n \alpha \Vert^{\frac{1}{\beta}}, \frac{%
1}{N} \Vert q_{n-1} \alpha \Vert]$ then $2r \cdot c \in [cM \Vert q_n \alpha
\Vert^{\frac{1}{\beta}}, \frac{1}{N/c} \Vert q_{n-1} \alpha \Vert]$.

Let $\ell_n:= \lceil \log \frac{\frac{1}{N_n} \Vert q_{n-1}
\alpha \Vert}{M_n\Vert q_n \alpha \Vert^{\frac{1}{\beta}}} \rceil$ be a bound to the
cardinality of $I_n$. We apply $\ell_n$ times lemma \ref{lem:keyl} with $%
M=M_n, M_n \cdot e, \ldots, M_n \cdot e^{\ell_n-1}$ and $N=N_n \cdot
e^{\ell_n-1}, N_n \cdot e^{\ell_n-2}, \ldots, N_n$: 
\begin{eqnarray*}
\sum_{i \in I_n} \mu(A_i) & \leq & \frac{1}{N_n} \left( \frac{1}{e^{\ell_n-1}%
} + \frac{1}{e^{\ell_n-2}} + \ldots + 1 \right) + \frac{1}{M_n^\gamma}
\left( 1 + \frac{1}{e^\gamma} + \ldots + \frac{1}{e^{\gamma (\ell_n-1)}}
\right) \\
&\leq & \frac{1}{N_n} \left( \frac{1}{1-e^{-1}} \right) + \frac{1}{M_n^\gamma%
} \left( \frac{1}{1-e^{-\gamma}} \right).
\end{eqnarray*}

This argument applies equivalently to primed sequence. $\gamma >1$ so $\frac{%
1}{M_{n}^{\gamma }},\frac{1}{{M_{n}^{\prime }}^{\gamma }}$ are summable.
Last step is to show summability of $\frac{1}{N_{n}},\frac{1}{N_{n}^{\prime }%
}$: 
\begin{equation*}
N_{n}=\frac{\Vert q_{n-1}\alpha \Vert }{M_{n-1}^{\prime }\Vert
q_{n-1}^{\prime }\alpha ^{\prime }\Vert ^{\frac{1}{\beta }}}\geq \frac{1}{n-1%
}\frac{{q_{n}^{\prime }}^{\frac{1}{\beta }}}{2q_{n}}\geq \frac{1}{2(n-1)}%
q_{n}^{\frac{\gamma }{\beta }-1}
\end{equation*}%
\begin{equation*}
N_{n}^{\prime }=\frac{\Vert q_{n-1}^{\prime }\alpha ^{\prime }\Vert }{%
M_{n}\Vert q_{n}\alpha \Vert ^{\frac{1}{\beta }}}\geq \frac{1}{n}\frac{%
q_{n+1}^{\frac{1}{\beta }}}{2q_{n}^{\prime }}\geq \frac{1}{2n}{q_{n}^{\prime
}}^{\frac{\gamma }{\beta }-1}
\end{equation*}%
But $\gamma /\beta -1>0$ and by Proposition \ref{prop:fibo} both $%
q_{n},q_{n}^{\prime }$ eventually grow faster than any power of $n$ so $%
N_{n},N_{n}^{\prime }\geq n^{\delta }$ for some $\delta >1$.
\end{proof}

\begin{remark}
\label{rem:chevallier} 
A statement similar to theorem \ref{thm:infliminf} follows from a result in
a paper by Bugeaud and Chevallier (\cite{BC}). From their Theorem 2 it
follows that for every function $\phi: \mathbb{N} \to \mathbb{R}^+$
tending to $0$, there
exist rationally independent  $\overline{\alpha}=(\alpha_1,\alpha_2) \in \mathbb{R}^2$ such that the
following subset of $\mathbb{R}^2$ has zero Lebesgue measure: 
\begin{equation*}
\mathcal{W}_\phi (\overline{\alpha}) = \left\{ (x_1,x_2) \in \mathbb{R}^2 :
\max_{i=1,2} \lbrace \Vert n \alpha_i - x_i \Vert \rbrace < \phi(\vert n
\vert) \text{ holds for infinitely many } n \in \mathbb{Z} \right\}.
\end{equation*}
With an appropriate choice of $\phi$ and passing to the complement, this
implies that for a fixed $\gamma$ there exists a vector $\overline{\alpha}$
such that the following subset of $\mathbb{T}^2$ has full Lebesgue measure: 
\begin{equation*}
\left\{ \overline{x} \in \mathbb{T}^2 : \frac{-\log \mathrm{dist}(T_{%
\overline{\alpha}}^{n}(\overline{x}),(0,0))}{\log n} \leq \frac{1}{\gamma} 
\text{ definitively for } n \in \mathbb{N} \right\}
\end{equation*}
By Propositions \ref{prop:htll} and \ref{prop:bosheq} this is equivalent to
the existence of a translation on the torus such that $\underline{R}(%
\overline{x},(0,0))\geq \gamma$ for Lebesgue almost every $\overline{x} \in 
\mathbb{T}^2$.
\end{remark}

\section{Reparametrization of flows and hitting time\label{fow}}

In this section we show that the time one map of the mixing flow on the $3$%
-torus constructed by Fayad (see Theorem \ref{fayadex}) has properties
similar to the ones described in Theorem \ref{thm:infliminf}. Fayad's
example is a flow, hence we also consider the concept of hitting time in the
context of continuous time dynamical systems:

\begin{definition}
Let $\Phi _{t}:X\rightarrow X$ be a flow on $X$. Take $x\in X$ and consider
the time needed for $x$ to enter in a ball $B_{r}(x_{0})$: 
\begin{equation*}
\tau _{r}^{(\Phi _{t})}(x,x_{0}):=\inf \{t>0\mid \Phi _{t}(x)\in
B_{r}(x_{0})\}.
\end{equation*}
\end{definition}

Working as in (\ref{expo}), this definition naturally provides a definition
of hitting time indicator for continous systems and which we will denote as $%
R^{(\Phi _{t})}$ (while the indicator of the discrete system given by a map $%
T$ will be denoted by $R^{(T)}$).

We now consider the relations between hitting time in flows and in
associated discrete time systems, as time-$1$ maps or Poincar\'{e} sections.
As a first trivial remark we note that the hitting time indicator of a flow
and its time-$1$ map can be quite different, hence we have to be careful to
specify when discrete or continuous time is considered. In fact, think of a
rotation on the circle by an irrational number $\alpha $ of type 1 and
consider the associated flow given by $\Phi _{t}(x)=x+t\alpha \pmod 1$ which
is such that $T_{\alpha }=\Phi _{1}$. In this case we have that for each
fixed $x_0$ 
\begin{equation*}
R^{(T_{\alpha })}(x,x_{0})=1\quad \mathit{while}\quad R^{(\Phi
_{t})}(x,x_{0})=0
\end{equation*}%
for almost every $x$. In general however we have the following immediate
relation between the flow and the time $1$ associated map

\begin{proposition}[Flow and time $1$ map]
\label{map1}If $\Phi _{t}$ is a flow, as above, we have that if $r$ is small
enough%
\begin{equation*}
\tau _{r}^{(\Phi _{1})}(x,y)\geq \tau _{r}^{(\Phi _{t})}(x,y)
\end{equation*}%
for every $x,y$ and hence%
\begin{equation*}
\overline{R}^{(\Phi _{1})}(x,y)\geq \overline{R}^{(\Phi _{t})}(x,y)~and~%
\underline{R}^{(\Phi _{1})}(x,y)\geq \underline{R}^{(\Phi _{t})}(x,y).
\end{equation*}
\end{proposition}

Now we want to consider the case of a Poincar\'{e} section of a
translation on the $d$-torus $\mathbb{T}^{d}=\mathbb{R}^{d}/\mathbb{Z}^{d}$.
We introduce some general notation which will be used in the following. We
equip $\mathbb{T}^{d}$ with the $\sup $ distance: if $x=(x_{1},\ldots
,x_{d}) $ and $y=(y_{1},\ldots ,y_{d})$ then $\mathrm{dist}%
(x,y)=\sup_{i}(\Vert x_{i}-y_{i}\Vert) $. A translation by a vector $\alpha
=(\alpha _{1},\ldots ,\alpha _{d})\in R^{d}$ on $\mathbb{T}^{d}$ is the
function $T_{\alpha }:\mathbb{T}^{d}\rightarrow \mathbb{T}^{d}$ given by 
\begin{equation*}
(x_{1},\ldots ,x_{n})\mapsto (x_{1}+\alpha _{1},\ldots ,x_{n}+\alpha _{n}).
\end{equation*}%
The translation is said to be irrational if the numbers $1,\alpha_{1},\ldots
,\alpha _{d}$ are rationally independent. In this case $T_{\alpha }$ is
uniquely ergodic and its invariant measure is the Lebesgue one.

Now let us consider the translation flow $\Phi _{t}$ on
$\mathbb{T}^{d}$ with direction $\alpha \in \mathbb{R}^{d}$. This is the flow generated integrating the
constant vector field $X(x)=\alpha $. As above, the flow is said to be
irrational if $\alpha _{1},\ldots ,\alpha _{d}$ are rationally independent.
Also in this case $\Phi _{t}$ is uniquely ergodic, and the
invariant measure is the Lebesgue one (\cite{fayad}).

Let $\Phi _{t}$ be a translation flow on $\mathbb{T}^{d}$ and let $%
\mathbb{T}^{d-1}$ be the torus of codimension $1$ given by $x_{d}=c$ (we
will consider on $\mathbb{T}^{d-1}$ the metric induced by the inclusion in $%
\mathbb{T}^{d}$). Let us consider the map $T_{\mathbb{T}^{d-1}}^{\Phi }:\mathbb{T}^{d-1}\rightarrow \mathbb{T}^{d-1}$ which is the Poincar\'{e} section of $%
\Phi _{t}$ on $\mathbb{T}^{d-1}$.  Let $\pi (x)=\Phi _{t_{0}} (x)$ with $t_{0}=\inf \{t>0\mid \Phi _{t}(x)\in \mathbb{T}%
^{d-1}\} $ be the "projection" on $\mathbb{T}^{d-1}$ induced by $\Phi
_{t}$.
 Then the following relation holds

\begin{proposition}[Flow and a section]
\label{sect}Let $T_{\mathbb{T}^{d-1}}^{\Phi }$ as above and $y\in \mathbb{T}%
^{d-1}$, let $x\in \mathbb{T}^{d}$. Then there are constants $K,C>0$
which depends also on the parametrization speed and the angle between flow and section, such that 
\begin{equation*}
\tau _{r}^{(\Phi _{t})}(x,y)\geq C\tau _{Kr}^{(T_{\mathbb{T}^{d-1}}^{\Phi
})}(\pi (x),y)
\end{equation*}%
for every $x$ such that $\pi (x)\neq y$, when $r$ is small enough. Then
under these assumptions%
\begin{equation*}
\overline{R}^{(\Phi _{1})}(x,y)\geq \overline{R}^{(T_{\mathbb{T}%
^{d-1}}^{\Phi })}(\pi (x),y)~and~\underline{R}^{(\Phi _{1})}(x,y)\geq 
\underline{R}^{(T_{\mathbb{T}^{d-1}}^{\Phi })}(\pi (x),y).
\end{equation*}
\end{proposition}

We remark that if a flow is irrational then also its Poincar\'e section, as
defined above is irrational.

Now we study invariance of continous hitting time indicator with respect to
a wide class of time reparametrizations. Given a vector $\alpha $ and a
striclty positive, smooth function $\phi :\mathbb{T}^{d}\rightarrow \mathbb{R%
}$ we can define the reparametrized translation flow $\Phi _{t}^{\phi
,\alpha }$ with velocity $\phi $, as the flow given by the vector field $%
\phi (x)\alpha $, that is by integrating the system 
\begin{equation*}
\dot{x}=\phi (x)\alpha .
\end{equation*}

The new flow has the same orbits as the original one and preserves a measure
which is absolutely continuous with respect to the Lesbegue measure having
density $\frac{1}{\phi (x)}.$ Moreover if the original flow is ergodic then
also the reparametrized one is. Under quite general assumptions a regular
reparametrizations $\phi $, does not change the hitting time behavior of the
flow.

\begin{proposition}[Reparametrization of a flow]
\label{reparametriz}If $\phi $ is $C^{1}$ and strictly positive on the
(compact) torus $\mathbb{T}^{d}$ as above then there exists a constant $%
C\geq 1$ such that $\frac{1}{C}<\phi (x)<C$ for all $x \in \mathbb{T}^{d}$.
Thus 
\begin{equation*}
C\tau _{r}^{(\Phi _{t})}(x,y)\geq \tau _{r}^{(\Phi _{t}^{\phi
,\alpha })}(x,y)\geq \frac{1}{C}\tau _{r}^{(\Phi _{t})}(x,y);
\end{equation*}
this implies 
\begin{equation*}
R^{(\Phi _{t})}(x,y)=R^{(\Phi _{t}^{\phi ,\alpha })}(x,y).
\end{equation*}
\end{proposition}

Combining the above propositions we obtain a way to estimate from below the
hitting time indicators of time-$1$ maps of reparametrizations of rotations
(or even other kind of maps constructed in a similar way) by the hitting
time of discrete time translations with the same angles.

\begin{proposition}
\label{DUE} Let $\Phi =\Phi _{1}^{(\phi ,\mathbf{\alpha })}$ be the time one
map of a bounded speed reparametrization of a translation flow on the
direction vector $\mathbf{\alpha }=(1,\alpha ,\alpha ^{\prime })$ and $T=T_{%
\mathbb{T}^{2}}^{(\alpha ,\alpha ^{\prime })}$ be the translation on the two
torus given by rationally independent angles $(\alpha ,\alpha ^{\prime })$, then for any $y\in 
\mathbb{T}^{3}$, $b\in \mathbb{T}^{2}$ 
\begin{equation*}
\overline{R}^{\Phi }(x,y)\geq \overline{R}^{T}(a,b)
\end{equation*}%
\begin{equation*}
\underline{R}^{\Phi }(x,y)\geq \underline{R}^{T}(a,b)
\end{equation*}%
for almost each $(a,x)$ (with respect to the Lesbegue measure in the product 
$\mathbb{T}^{2}\times \mathbb{T}^{3}$).
\end{proposition}

\begin{proof}
Since $T_{\mathbb{T}^{2}}^{(\alpha ,\alpha ^{\prime })}$ is ergodic,
and $\mu (b)=0$, then $\overline{R}^{T}(a,b)$ is almost everywhere constant when $a$ varies. This
constant moreover is obviously the same for each $b$ (see proposition \ref{inizz}). We will show that $\overline{R}^{\Phi }(x,y)$ is a.e. greater than
this constant when $x$ varies. The same can be done for $\underline{R}^{\Phi
}(x,y)$.

By the assumption, $\Phi _{1}^{(\phi ,\mathbf{\alpha })}$ is a time 1 map of
a flow on  $\mathbb{T}^{3}$ which is the reparametrization of a
translation flow $\Phi _{t}$ with direction $\mathbf{\alpha }=(1,\alpha ,\alpha ^{\prime })$.
 
If $y=(y_{1},y_{2},y_{3})\in \mathbb{T}^{3}$, we can consider the Poincar\'{e}   section $T_{\mathbb{T}^{2}}^{(\alpha ,\alpha ^{\prime })}$  of this
flow on a $2$-torus $\mathbb{T}^{2}$ which is the set of points in $\mathbb{T}%
^{3}$ whose first coordinate is $y_{1}$ (this torus hence contains $y$).
Then, combining Propositions \ref{map1}, \ref{sect} and \ref{reparametriz}
there is a constant $K,C>0$ such that for any $x,y\in \mathbb{T}^{3}$, when $r$
is small%
\begin{equation*}
\tau _{r}^{\Phi _{1}^{(\phi ,\mathbf{\alpha })}}(x,y)\geq C\tau _{Kr}^{T_{{%
\mathbb{T}}^{2}}^{(\alpha ,\alpha ^{\prime })}}(\pi (x),y)
\end{equation*}%
where $\pi (x)$ is the projection defined before Proposition \ref{sect}. The
statement follows, because $\pi ^{-1}(A-\{y\})$ is a full measure subset of $%
\mathbb{T}^{3}$ when $A$ is a full measure set in $\mathbb{T}^{2}$.
\end{proof}

In Proposition \ref{reparametriz} we have seen that reparametrizing a flow,
does not change the hitting time indicators. On the other side a flow
reparametrizations can have surprisingly different ergodic properties
compared to the original flow. A reparametrization of a translation can be
mixing as we will see below.

Let $Y$ be the set of couples $(\alpha _{1},\alpha _{2})\in \mathbb{R}^{2}-%
\mathbb{Q}^{2}$ such that the respective best approximant denomiators $%
q_{n},q_{n}^{\prime }$ satisfies definitively 
\begin{equation*}
q_{n}^{\prime }\geq e^{3q_{n}}
\end{equation*}
\begin{equation*}
q_{n+1}\geq e^{3q_{n}^{\prime }}.
\end{equation*}%
Note that $Y$ is an uncountable and dense set and $Y \subset Y_\infty$.

The main result of \cite{fayad} states:

\begin{theorem}
(Fayad's example) \label{fayadex}There is a stricly positive analytic
function $\phi $ on $\mathbb{T}^{3},$ such that for any $(\alpha _{1},\alpha
_{2})\in Y$ the reparametrization with speed $\frac{1}{\phi }$ of the
irrational flow given by the vector $(\alpha _{1},\alpha _{2},1)$ is mixing.
\end{theorem}

We remark that the particular choice of numbers $\alpha _{1}$ and $\alpha
_{2}$ given by the above equations is important. It is proved (\cite{fayad},
Thm. 2) that for a generic choice of $\alpha _{1}$ and $\alpha _{2}$ we
cannot have a similar result.
By Proposition \ref{DUE} and Theorem \ref{thm:infliminf} the time-1 map of
this system has infinite hitting time indicator. This will provide an
example of a smooth mixing system with no relation between hitting time and measure
(in other words a smooth, mixing system with no logarithm law).

\begin{corollary}
\label{thm:fydconn} 
Let $(\mathbb{T}^{3},\Phi _{1}^{\phi ,\alpha })$ be the time $1$ map of the
system described in Theorem \ref{fayadex} equipped with its natural
absolutely continuous invariant measure $\mu $. This system is mixing and we
have that for each $x$, $d_{\mu }(x)=3,$ while%
\begin{equation*}
\overline{R}(y,x)=\underline{R}(y,x)=\infty
\end{equation*}%
for almost each $y.$
\end{corollary}

\begin{proof}
This map is the time-1 map of a strictly positive speed reparametrization of
a translation flow $(\mathbb{T}^{3},\Phi _{t})$ with angle $\mathbf{\alpha }=(1,\alpha ,\alpha ^{\prime })$ and $(\alpha ,\alpha
^{\prime })\in Y\subset Y_{\infty }$. By proposition \ref{DUE} its hitting
time indicator is greater or equal than the one given by the translation $T_{%
{\mathbb{T}}^{2}}^{(\alpha ,\alpha ^{\prime })}$ on $\mathbb{T}^{2}$ given
by the angles $(\alpha ,\alpha ^{\prime })$ and this, by Theorem \ref%
{thm:infliminf} are infinite.
\end{proof}

\section{Decay of correlations and hitting time}

\label{lastsection}

It is well known that many kinds of chaotic dynamics ''mixes'' the
phase space. Decay of correlation speed gives a quantitative
estimation for the speed of this mixing behavior. We recall the definition.

\begin{definition}
\label{sup} A system $(X,T,\mu )$, where $T$ is $\mu$-invariant,
is said to have decay of correlations with
power law speed with exponent $p$ if there exists $\Phi(n)$ such that, for
all Lipschitz observables $\phi ,$ $\psi :X\rightarrow \mathbb{R}$ on $X$ 
\begin{equation*}
|\int \phi \circ T^{n}\psi d\mu -\int \phi d\mu \int \psi d\mu |\leq
\left\vert \left\vert \phi \right\vert \right\vert \left\vert \left\vert
\psi \right\vert \right\vert \Phi (n).
\end{equation*}%
and $\lim_{n\rightarrow \infty }\frac{-\log \Phi (n)}{\log n}=p.$ Here $%
||~|| $ is the Lipschitz norm\footnote{
This is a definition for Lipschitz observables which is one of the
weakest possible statements. In many systems decay of correlation is
proved for observables in other functional spaces with weaker norms (Holder functions e.g.). \ In
this cases the decay of correlation for Lipschitz observables follows by
estimating the weaker norms by the Lipschitz one.}.
\end{definition}

We remark that by the above definition, decay of correlations slower than
any negative power law (logarithmic decay for example) are considered as
power laws with zero exponent.

\begin{lemma}
\label{uno} Let $A_{n}=T^{-n}(B_{r}(x_{0}))$. If $(X,T,\mu )$ is a system
satisfying definition \ref{sup} then for each small $\lambda >0$%
\begin{equation}
\mu (A_{k}\cap A_{j})\leq \mu (B_{(\lambda +1)r}(x_{0}))^{2}+\frac{4\Phi (k-j)%
}{\lambda^{2}r^{2}}  \label{ball}
\end{equation}
holds eventually as $r\to 0 $.
\end{lemma}

\begin{proof}
Let $\phi _{\lambda }$ be a Lipschitz function with norm less than $\frac{2}{\lambda r}
$such that $\phi _{\lambda }(x)=1$ for all $x\in B_{r}(x_{0})$, $\phi
_{\lambda }(x)=0$ if $x\notin B_{(1+\lambda )r}(x_{0})$, the result follows
directly by definition \ref{sup}, by the remarking that by invariance of $%
\mu $ 
\begin{equation}
\mu (A_{k}\cap A_{j})\leq \int \phi _{r}\circ T^{k-j}\phi _{r}d\mu .
\end{equation}
\end{proof}

We now prove that in systems with at least polynomial decay of correlations
and "good" measure the hitting time in small balls is related to the inverse
of the ball measure up to an error which is controlled by decay of
correlation speed.

\begin{theorem}\label{1}
Let $(X,T,\mu)$ be a (probability)  measure preseving transformation
 on a metric space.  Suppose that for the point $x_0$ we have $\mu (\{x_0\})=0$, $0<\underline{d}_{\mu }(x_{0})\leq \overline{d}_{\mu }(x_{0})<\infty $ and the
measure $\mu $ satisfies the following property\footnote{
Which roughly means that the measure goes polynomially to zero for every
annulus in a neighbourhood of $x_0$; this is due to the technical
requirement that we assume decay of correlations of observables but we want
to use it for events (see Lemma \ref{uno}).}: there exist $C,\beta >0$ such
that for each $r,\lambda >0$ small enough 
\begin{equation}
\mu (B_{(1+\lambda )r}(x))\leq \mu (B_{r}(x))(1+C\lambda ^{\beta }).
\label{doub}
\end{equation}
If the system has polynomial decay of correlation with exponent $p$, then 
\begin{equation*}
\underset{r\rightarrow 0}{\lim \sup }\frac{\log \tau _{r}(x,x_{0})}{-\log
\mu (B_{r}(x_{0}))}\leq 1+\frac{2\overline{d}_{\mu }(x_{0})+2}{\underline{d}%
_{\mu }(x_{0})p}
\end{equation*}
holds for $\mu $-almost each $x$.
\end{theorem}

\begin{proof}
(of Theorem C) Lower bound follows
from Proposition \ref{GAN}.  An absolutely continuous invariant measure on a manifold of
dimension $d$ having continuous and strictly positive density satisfies the
hypotheses \ref{doub} on the measure at every point. In addition it is exact
dimensional with $\underline{d}_\mu=\overline{d}_\mu=d$. Upper bound follows from Theorem \ref{1} (cfr.
also Remark \ref{rem:tausumu}).
\end{proof}

\begin{proof}
(of Theorem \ref{1}) Since we consider a ratio of logarithms and $\overline{d%
}_{\mu }(x_{0})<\infty $, the measure $\mu (B_{e^{-n}}(x_{0}))$ will
decrease at most exponentially fast, and without loss of generality we can
restrict to a sequence of radii $r_{n}=e^{-n}$ (see Lemma \ref{lemmino}).

Set $\zeta=\frac{2\overline{d}_{\mu }(x_{0})+2}{p}+\epsilon$, $\epsilon >0$
and 
\begin{equation*}
A_{n}=\underset{1\leq i\leq e^{\zeta n}\left\lfloor \mu
(B_{r_{n}}(x_{0}))^{-1}\right\rfloor }{\cup }T^{-i}(B_{r_{n}}(x_{0})).
\end{equation*}%
To prove our thesis it is sufficient to prove the following statement: for
every $\zeta>\frac{2\overline{d}_{\mu }(x_{0})+2}{p}$, $\mu$-almost every $x
\in X$ belongs eventually to $A_{n}$.
Indeed we have definitively $\tau _{r_n}(x,x_{0})\leq e^{\zeta
n}\left\lfloor \mu (B_{r_{n}}(x_{0}))^{-1}\right\rfloor $, then 
\begin{equation*}
\frac{\log (\tau _{r_{n}}(x,x_{0}))}{-\log \mu (B_{r_{n}}(x_{0}))}\leq \frac{%
\zeta n+\log \left\lfloor \mu (B_{r_{n}}(x_{0}))^{-1}\right\rfloor }{-\log
\mu (B_{r_{n}}(x_{0}))}\leq \frac{2\overline{d}_{\mu }(x_{0})+2+\epsilon p}{(%
\underline{d}_{\mu }(x_{0})-\epsilon )p}+\frac{\log \left\lfloor \mu
(B_{r_{n}}(x_{0}))^{-1}\right\rfloor }{\log \mu (B_{r_{n}}(x_{0}))^{-1}}.
\end{equation*}
And the thesis follows by letting $n\to \infty$ and $\epsilon \to 0$.

Hence, we set $m_{n}=e^{\zeta n}\left\lfloor \mu
(B_{r_{n}}(x_{0}))^{-1}\right\rfloor$ and consider 
\begin{equation*}
Z_{n}(x)=\sum_{i=1}^{m_{n}}1_{T^{-i}(B_{r_{n}}(x_{0}))}(x).
\end{equation*}%
If $Z_{n}(x)>0$ then $x\in A_{n}$. If we prove that $\frac{Z_{n}(x)}{\mathbf{%
E}(Z_{n})}\rightarrow 1$ almost everywhere for every $\zeta>\frac{2\overline{%
d}_{\mu }(x_{0})+2}{p}$ then the statement is proved.

The idea is to estimate $\mathbf{E}((Z_{n})^{2})$ and find an upper bound
which ensures that the distribution of the possible values of $Z_{n}$ is not
too far from the average $\mathbf{E}(Z_{n})$. To compare $Z_{n}$ with its
average $\mathbf{E}(Z_{n})$ we consider 
\begin{equation*}
Y_{n}=\frac{Z_{n}}{\mathbf{E}(Z_{n})}-1=\frac{Z{_{n}}-\mathbf{E}{(Z_{n})}}{%
\mathbf{E}(Z_{n})}.
\end{equation*}%
When $Y_{n}=0$, $Z_{n}=\mathbf{E}(Z_{n})$, thus we need to prove that $Y_n
\to 0$ almost everywhere.

\bigskip We denote the preimages of the balls as $%
B_{k}=T^{-k}(B_{r_{n}}(x_{0})).$

We have that 
\begin{equation}  \label{eq:ce}
\mathbf{E}((Z_{n})^{2})=\sum_{k=1}^{m_{n}}\mu (B_{k})+2\sum_{k,j\leq
m_{n},k>j}\mu (B_{k}\cap B_{j}).
\end{equation}

The second summand on the right side of (\ref{eq:ce}) can be thought as the
sum of the off-diagonal elements of a matrix with entries $\mu (B_{k}\cap
B_{j})$. We split this sum in two parts by considering separately the
entries `near' or `far' the diagonal. We measure this `nearness' in the
following way: take $\alpha$ such that $\frac{2\overline{d}_{\mu }(x)+2}{p}%
<\alpha <\zeta$ and split the sum as 
\begin{equation*}
\sum_{k,j\leq m_{n},k>j}\mu (B_{k}\cap B_{j}) =
\end{equation*}%
\begin{equation}  \label{eq:splitsum}
= \sum_{k,j\leq m_{n},k>j,k<j+e^{\alpha n}}\mu (B_{k}\cap
B_{j})+\sum_{k,j\leq m_{n},k\geq j+e^{\alpha n}}\mu (B_{k}\cap B_{j}).
\end{equation}

Since $\mu (B_{k}\cap B_{j})\leq \mu (B_{k})$, the first summation in (\ref%
{eq:splitsum}) can be largely estimated as follows: 
\begin{equation}  \label{eq:firstsum}
\sum_{k,j\leq m_{n},k>j,k<j+e^{\alpha n}}\mu (B_{k}\cap B_{j})\leq e^{\alpha
n}\mathbf{E}(Z_{n}).
\end{equation}

To estimate the second summation we use hypothesis on speed of decay of
correlations. From Lemma \ref{uno}, taking $\lambda =n^{-\frac{2}{\beta }}$%
(the value of $\beta $ depends on $\mu $ as given in (\ref{doub})) we have
that 
\begin{equation}
\mu (B_{k}\cap B_{j})\leq \mu (B_{(1+n^{-\frac{2}{\beta }%
})r_{n}}(x_{0}))^{2}+\frac{4\Phi (k-j)}{(n^{-\frac{2}{\beta }})^{2}r_{n}^{2}}.
\label{mixxx}
\end{equation}%
with $\Phi $ as in definition \ref{sup} is polynomially decaying with
exponent $p$. By equation (\ref{doub}) we have that 
\begin{equation}  \label{eq:newlab}
\mu (B_{(1+n^{-\frac{2}{\beta }})r_{n}}(x_{0}))^{2}\leq \mu
(B_{r_{n}}(x_{0}))^{2}(1+Cn^{-2})^{2}.
\end{equation}
Using the bounds in (\ref{mixxx}) and (\ref{eq:newlab})
we obtain%
\begin{equation*}
\sum_{k,j\leq m_{n},k\geq j+e^{\alpha n}}\mu (B_{k}\cap B_{j})\leq
\end{equation*}%
\begin{equation*}
\leq \sum_{k,j\leq m_{n},k\geq j+e^{\alpha n}} \left[ \mu (B_{(1+n^{-\frac{2%
}{\beta }})r_{n}}(x_{0}))^{2}+\frac{4\Phi (k-j)}{(n^{-\frac{2}{\beta }%
})^{2}r_{n}^{2}} \right] \leq
\end{equation*}%
\begin{equation*}
\leq \frac{1}{2}\mathbf{E}((Z_{n}))^{2}(1+Cn^{-2})^{2}+\frac{(m_{n})^{2}}{%
(n^{-\frac{2}{\beta }})^{2}(r_{n})^{2}} 4 \Phi (e^{\alpha n}).
\end{equation*}%
Hence by (\ref{eq:firstsum})
\begin{equation*}
\sum_{k,j\leq m_{n},k>j}\mu (B_{k}\cap B_{j})\leq
\end{equation*}%
\begin{equation}  \label{eq:secsum}
\leq e^{\alpha n}\mathbf{E}(Z_{n})+\frac{1}{2}(\mathbf{E}%
(Z_{n}))^{2}(1+Cn^{-2})^{2}+\frac{4(m_{n})^{2}}{(n^{-\frac{2}{\beta }%
})^{2}(r_{n})^{2}}\Phi (e^{\alpha n})
\end{equation}


Now, plugging in (\ref{eq:firstsum}) and (\ref{eq:secsum}) into (\ref{eq:ce}%
), we obtain 
\begin{eqnarray*}
\mathbf{E}((Z_{n}-\mathbf{E}(Z_{n}))^{2}) & = & \mathbf{E}((Z_{n})^{2})-(%
\mathbf{E}(Z_{n}))^{2} \leq \\
& \leq &(2 e^{\alpha n}+1)\mathbf{E}(Z_{n})+(\mathbf{E}%
(Z_{n}))^{2}(C^{2}n^{-4}+2Cn^{-2})+ \\
& & +\frac{4 (m_{n})^{2}}{(n^{-\frac{2}{\beta }})^{2}(r_{n})^{2}}\Phi
(e^{\alpha n}).
\end{eqnarray*}
which, in terms of $Y_n$, amounts to 
\begin{eqnarray*}
\mathbf{E}((Y_{n})^{2}) &\leq &\frac{(e^{\alpha n+1}+1)\mathbf{E}(Z_{n})+(%
\mathbf{E}(Z_{n}))^{2}(C^{2}n^{-4}+2Cn^{-2})}{(\mathbf{E}(Z_{n}))^{2}}+ \\
&&+\frac{4 (m_{n})^{2}}{(\mathbf{E}(Z_{n}))^{2}(n^{-\frac{2}{\beta }%
})^{2}(r_{n})^{2}}\Phi (e^{\alpha n})
\end{eqnarray*}%
and, since $\mathbf{E}(Z_{n}) = m_{n} \cdot \mu (B_{r_{n}}(x_{0}))= e^{\zeta
n} \left\lfloor \mu (B_{r_{n}}(x_{0}))^{-1}\right\rfloor
\mu(B_{r_{n}}(x_{0}))$, we have 
\begin{equation*}
\mathbf{E}((Y_{n})^{2}) \leq \frac{(e^{\alpha n+1}+1)}{m_{n} \, \mu
(B_{r_{n}}(x_{0}))}+ \frac{4 \Phi(e^{\alpha n})}{\mu
(B_{r_{n}}(x_{0}))^{2}(n^{-\frac{2}{\beta }})^{2}(r_{n})^{2}}%
+C^{2}n^{-4}+2Cn^{-2}
\end{equation*}
Finally, we note that $\zeta>\alpha $ and $p \, \alpha >2\overline{d}_{\mu
}(x)+2$, thus $\mathbf{E}((Y_{n})^{2})$ goes to zero in a summable\footnote{%
We remark that $\frac{\Phi (e^{\alpha n})}{\mu (B(x,r_{n}))^{2}(n^{-\frac{2}{%
\beta }})^{2}(r_{n})^{2}} \leq C^{\prime}\, n^{4/\beta } \frac{e^{-p \alpha
n}}{e^{-n (2\overline{d}+2)}} = n^{4/\beta } e^{-n[p\alpha - (2\overline{d}%
+2)]}$.} way. This proves that $\frac{Z_{n}}{\mathbf{E}(Z_{n})}\rightarrow 1$
almost everywhere.
\end{proof}

\begin{corollary}
\label{corD} Under the assumptions of the previous theorem, if $T$  power law decay of correlations then
\begin{equation*}
\overline{R}(x,x_{0})<\infty ,\underline{R}(x,x_{0})<\infty
\end{equation*}%
$\mu $ for a.e. $x\in X$.
\end{corollary}

This gives us that no reparametrizations of translations with infinite
hitting time indicator (as in Fayad's example, as showed in section 4
and 5) can be polynomially
mixing, moreover.

\begin{corollary}
\label{cormanifold} If $(X,T)$ is a map on an $d$ dimensional manifold,
having an absolutely continuous invariant measure
with strictly positive density and $\overline{R}(x,y)=R$
a.e., then the speed of decay of correlations of the system is a power law
with exponent $p\leq{\frac{2d+2}{R-d}}$.
\end{corollary}

\begin{proof}
(of Theorem D) Follows from Proposition \ref{reparametriz}, Proposition \ref{UNO} and
Corollary \ref{cormanifold}.
\end{proof}

\section{Appendix}

In this section we recall and precise some results on the equivalence
between different approaches to the hitting time problem.

Let us denote by $\mathrm{dist}(\cdot ,\cdot )$ the distance on $X$ and
define $d_{n}(x,y)=\min_{1 \leq i\leq n}\mathrm{dist}(T^{i}(x),y)$.

\begin{proposition}
\label{prop:htll} 
Given a system $T$ on a metric space $(X,\mathrm{dist})$: 
\begin{equation*}
\underline{R}(x,x_{0})=\left( \limsup_{n\rightarrow \infty}\frac{-\log
d_{n}(x,x_{0})}{\log n}\right) ^{-1}
\end{equation*}%
and 
\begin{equation*}
\overline{R}(x,x_{0})=\left( \liminf_{n\rightarrow \infty}\frac{-\log
d_{n}(x,x_{0})}{\log n}\right) ^{-1}.
\end{equation*}
\end{proposition}

\begin{proof}
Note that $d_n \leq r$ if and only if $\tau_r \leq n$. Suppose $%
\limsup_{n\to \infty} \frac{-\log d_n}{\log n}=a$ and take $\epsilon >0$.

There exist infinitely many $n$ such that $\frac{-\log d_n}{\log n} \geq
a-\epsilon$ that is $d_n \leq n^{-a+\epsilon}$, thus $\tau_{n^{-a+\epsilon}}%
\leq n$. Put $r=n^{-a+\epsilon}, \; n= r^{\frac{-1}{a -\epsilon}}$ and you
find a sequence of radii going to zero such that $\tau_r \leq r^{\frac{-1}{a
-\epsilon}}$, therefore $\liminf_{r\to 0} \frac{\log \tau_r}{-\log r} \leq 
\frac{1}{a}$.

Eventually $\frac{-\log d_{n}}{\log n}\leq a+\epsilon $ or $d_{n}\geq
n^{-a-\epsilon }$, thus $\tau _{n^{-a-\epsilon }}\geq n$. Putting $%
r=n^{-a-\epsilon },\;n=r^\frac{-1}{a+\epsilon }$ we obtain a sequence of radii
(decreasing slower than exponentially) for which $\tau _{r}\geq r^{\frac{-1}{%
a+\epsilon }}$, therefore $\liminf_{r\rightarrow 0}\frac{\log \tau _{r}}{%
-\log r}\geq \frac{1}{a}$.

This establishes a bijection ($a\mapsto \frac{1}{a}$) between the
non-negative quantities $\underline{R}(x,x_{0})$ and $\limsup_{n\rightarrow
\infty }\frac{-\log d_{n}}{\log n}$. The other equation is treated similarly.
\end{proof}

\begin{proposition}
\label{prop:didienne} For any function $f:\mathbb{N}\rightarrow \mathbb{R}$ 
\begin{equation*}
\sup \{\beta :\liminf_{n}n^{\beta }f(n)=0\}=\inf \{\beta
:\liminf_{n}n^{\beta }f(n)=\infty \}=\limsup_{n}\frac{-\log f(n)}{\log n}
\end{equation*}%
and 
\begin{equation*}
\sup \{\beta :\limsup_{n}n^{\beta }f(n)=0\}=\inf \{\beta
:\limsup_{n}n^{\beta }f(n)=\infty \}=\liminf_{n}\frac{-\log f(n)}{\log n}.
\end{equation*}
\end{proposition}

\begin{proof}
We will prove the first one, the other being similar. The equality between $%
\sup $ and $\inf $ is obvious. We will show: 
\begin{equation*}
\liminf_{n\rightarrow \infty }n^{\beta }f(n)=0\Rightarrow
\limsup_{n\rightarrow \infty }\frac{-\log f(n)}{\log n}\geq \beta ,
\end{equation*}%
\begin{equation*}
\liminf_{n\rightarrow \infty }n^{\beta }f(n)=\infty \Rightarrow
\limsup_{n\rightarrow \infty }\frac{-\log f(n)}{\log n}\leq \beta .
\end{equation*}

Take $\beta$ such that $\liminf_n n^\beta f(n) =0$ holds and $\epsilon>0$.
There exist infinitely many $n$ such that 
\begin{equation*}
n^\beta f(n) < \epsilon \Leftrightarrow \beta \log n + \log f(n) < \log
\epsilon \Leftrightarrow \frac{-\log f(n)}{\log n} > \beta - \frac{\log
\epsilon}{\log n}
\end{equation*}
thus $\limsup_n \frac{-\log f(n)}{\log n} \geq \beta$.

Take $\beta$ such that $\liminf_n n^\beta f(n) =\infty$ holds. For every $%
\epsilon>0$ definitively 
\begin{equation*}
n^\beta f(n) > \frac{1}{\epsilon} \Leftrightarrow \beta \log n + \log f(n) >
-\log \epsilon \Leftrightarrow \frac{-\log f(n)}{\log n} < \beta + \frac{%
\log \epsilon}{\log n}
\end{equation*}
thus $\limsup_n \frac{-\log f(n)}{\log n} \leq \beta$.
\end{proof}

\begin{proposition}
\label{prop:bosheq} 
\begin{equation}
\liminf_{n}n^{\beta }d_{n} (x,y) =0 \Leftrightarrow
\liminf_{n} n^{\beta }\mathrm{dist}(T^{n}(x),y)=0  \label{eq:bosheq}
\end{equation}%
thus 
\begin{equation*}
\limsup_{n}\frac{-\log d_{n}(x,y)}{\log n}= \limsup_{n}\frac{-\log \mathrm{%
dist}(T^{n}(x),y)}{\log n}.
\end{equation*}
\end{proposition}

\begin{proof}
To prove double implication we first observe that
$0 \leq n^{\beta }d_{n} (x,y) \leq n^{\beta }\mathrm{dist}(T^{n}(x),y)$, which gives us
one implication.
Now take a subsequence $\lbrace n_k \rbrace_{k \in \mathbb{N}}$ such that $\lim_k {n_k}^\beta d_{n_k} (x,y) = 0$.
For each $k$ we construct another subsequence $\lbrace n_k' \rbrace_{k \in \mathbb{N}}$
taking $n_k'$ the biggest integer such that $d_{n_k} (x,y) = \mathrm{dist}(T^{n_k'}(x),y)$
and $n_k' \leq n_k$. We have that $n_k' \to \infty$ and
${n_k'}^\beta \mathrm{dist}(T^{n_k'}(x),y) \leq {n_k}^\beta d_{n_k} (x,y) \to 0$.

Second part of the statement follows for (\ref{eq:bosheq}) and proposition \ref{prop:didienne}.
\end{proof}

\begin{remark} \label{rem:approcci}
Concerning relations between the two approaches of waiting time and Borel-Cantelli
we recall some results proved in \cite{GK}.
A result like equation (\ref{BC}) implies that, for every $x_0$ and almost every $x$,
\begin{equation}
\underset{r\rightarrow 0}{\lim }\frac{\log \tau _{r}(x,x_{0})}{-\log \mu
(B_{r}(x_{0}))}=1,
\end{equation}%
hence 
\begin{equation}
\lim_{r\rightarrow 0}\frac{\log \tau _{r}(x,x_{0})}{-\log r}=d_{\mu }(x_{0})
\end{equation}%
when $\mu (x_{0})=0$ and the local dimension exists.

Conversely there are systems for which $\liminf_{r\rightarrow 0}\frac{\log
\tau _{r}(x,x_{0})}{-\log r}=d_{\mu }(x_{0})$ but there are certain
decreasing sequences of balls with $s_{n}=\sum_{i=1}^{n}\mu
(B_{i})\rightarrow \infty $ such that $\mu (\limsup_{i}T^{-i}(B_{i}))=0$.
This situation, however, cannot happen if we restrict to balls where
$s_{n}\sim n^{\alpha }$ , $\alpha >0$.
\end{remark}

\end{document}